\documentclass[preprint,11pt,3p]{elsarticle}

\usepackage{color}
\usepackage[colorlinks=true]{hyperref}

\usepackage{amsmath,amssymb,amsthm,amscd}
\usepackage{times}
\usepackage{mathrsfs}
\usepackage[title]{appendix}
\numberwithin{equation}{section}

\newtheorem{theorem}{Theorem}[section]

\newtheorem{proposition}{Proposition}[section]
\newtheorem{lemma}{Lemma}[section]
\newtheorem{definition}{Definition}[section]
\newtheorem{corollary}{Corollary}[section]
\newtheorem{remark}{Remark}[section]

\renewcommand\appendix{\par
    \setcounter{section}{0}
    \setcounter{subsection}{0}
    \gdef\thesection{Appendix \Alph{section}}}

\biboptions{sort&compress}

\journal{Elsevier}

\begin{document}
\begin{frontmatter}
\title{Liouville Theorem for Lane-Emden Equation of Baouendi-Grushin operators}
\author{Hua Chen\corref{cor1}}
\ead{chenhua@whu.edu.cn}
\author{Xin Liao}
\ead{xin_liao@whu.edu.cn}

\address{School of Mathematics and Statistics, Wuhan University, Wuhan 430072, China}

\date{\today}
\begin{abstract}
In this paper, we establish a Liouville theorem for solutions to the Lane-Emden equation involving Baouendi-Grushin operators:
$$-(\Delta_{x} u+(\alpha+1)^2|x|^{2 \alpha} \Delta_{y} u) = |u|^{p-1}u,$$  
where $(x,y)\subset \mathbb{R}^N = \mathbb{R}^m \times \mathbb{R}^n$ with $m \geq 1$, $n\geq 1$, and  $\alpha \geq 0$.
We focus on solutions that are stable outside a compact set. Specifically, we prove that for
$p>1$, when $p$ is smaller than the Joseph–Lundgren exponent and differs from the Sobolev exponent, $u=0$ is the  unique $C^2$ solution  stable outside a compact set. This work extends the results obtained by Farina (J. Math. Pures Appl., \textbf{87} (5) (2007), 537–561).
\end{abstract}
\end{frontmatter}
\section{Introduction}
In this paper, we explore the following  Lane-Emden equation  of Baouendi-Grushin operators:
\begin{equation}\label{01}
    -\Delta_{\alpha} u(x,y)  = |u|^{p-1}u(x,y),~~ \forall (x,y)\subset \mathbb{R}^N = \mathbb{R}^m \times \mathbb{R}^n.
\end{equation}
Here,  $\alpha\geq 0, p>1$ and $\Delta_{\alpha}:=\Delta_{x} +(\alpha+1)^2|x|^{2 \alpha} \Delta_{y}$ denotes the  Baouendi-Grushin operators (cf. \cite{B1967}, \cite{G1970}).
When $\alpha=0$, \eqref{01}  reduces to the  conventional Lane-Emden equation:
\begin{equation} \label{02}
-\Delta u = |u|^{p-1}u~~ \text{in}~ \mathbb{R}^N.
\end{equation}
For equation \eqref{02}, there exist two critical exponents: the first one is the Sobolev exponent $p_{\text{S}}(N) = \frac{N+2}{N-2}$, and the second one is the Joseph–Lundgren exponent:
\begin{align*}
    p_{\text{JL}}(N) &=
    \begin{cases}
        +\infty, & \text{if }  N\leq 10, \\
        \frac{(N-2)^2-4N+8\sqrt{N-1}}{(N-2)(N-10)}, & \text{if } N\geq 11.
    \end{cases}
\end{align*}

In the pioneering work of Farina \cite{Farina}, a comprehensive classification of solutions for the conventional Lane-Emden equation was provided. Specifically, it was proved that:
\begin{enumerate}
\item For $1<p<p_{\text{JL}}(N)$, the unique $C^2$ stable solution to equation \eqref{02} is $u=0$.
\item For $p\in (1, p_{\text{S}}(N))\cup (p_{\text{S}}(N), p_{\text{JL}}(N))$, the unique $C^2$ solution to equation \eqref{02} that is stable outside a compact set is $u=0$.
\item For $p=p_{\text{S}}(N)$, equation \eqref{02} becomes the Yamabe equation. Due to the Cwikel-Lieb-Rozenblum inequality for the Schrödinger equation, Talent bubbles emerge as $C^2$ solutions with finite Morse index, making them stable outside a compact set.
\item For $N\geq 11$ and $p\geq p_{\text{JL}}(N)$,  standard phase plane analysis reveals that equation \eqref{02} admits positive radial $C^2$ stable solutions. This stability can be deduced from the Hardy inequality.
\end{enumerate}

Later, D\'{a}vila-Dupaigne-Wang-Wei \cite{Wang} provided a monotonicity formula and proved a Liouville theorem for solutions of the biharmonic Lane-Emden equation, both stable and stable outside a compact set. Additionally, Luo-Wei-Zou \cite{Luo} studied the polynomial harmonic case, while D\'{a}vila-Dupaigne-Wei \cite{Dupaigne} and Fazly-Wei \cite{Fazly} addressed the fractional and higher-order fractional cases.

For $n\geq1$ and $\alpha>0,$ the Baouendi-Grushin operator is known to be degenerate elliptic operator with $C^{\alpha}$ coefficients.
Denote $ Q=m+ n(\alpha +1)$ by the ‘‘homogeneous
dimension" of $\mathbb{R}^N$. Using the methods developed in \cite{Farina},  Rahal  \cite{Rahal}  proved that  any    $C^2$ stable  solution of \eqref{01}
is trivial if $p$ is smaller than $p_{\text{JL}}(Q)$.  However, to the best of my knowledge, no results have yet addressed solutions that are stable outside a compact set. Returning to our main problem, we present the following main result:

\begin{theorem}\label{thm1}\hfill
\begin{enumerate}[$(1)$]
  \item If  $p\in (1, p_{\text{S}}(Q))$, $u=0$ is the unique $C^2$ solution to \eqref{01} that is stable outside a compact set.
 \item For $p\in (p_{\text{S}}(Q), p_{\text{JL}}(Q))$, $u=0$ is the unique $C^2$ solution to  \eqref{01} that is stable outside a compact set.
\end{enumerate}
\end{theorem}

In the special case where $n=1, m=2k ~(k\in \mathbb{N}^*)$ and $\alpha=1,$ the equation \eqref{01}
 is associated with the Lane-Emden equation of the Heisenberg Laplacian in $\mathbb{H}^k:=\mathbb{R}^{2k} \times \mathbb{R}$.   For  $(x,y)$ and $(x',y')$ in $\mathbb{R}^{2k} \times \mathbb{R}$, the Heisenberg group multiplication operation $\circ$ is defined as:
\[(x,y) \circ (x',y') := (x+x', y+y'+2\sum_{i=1}^{k}(x_{i+k}x_i'-x_{i}x_{i+k}') ). \]
The left-invariant vector fields are given by:
\begin{equation}
X_{i}=\frac{\partial}{\partial x_i}+2x_{i+k}\frac{\partial}{\partial y},~
X_{i+k}=\frac{\partial}{\partial x_{i+k}}-2x_{i}\frac{\partial}{\partial y},~ i=1,\ldots, k.
\end{equation}
The Heisenberg Laplacian is defined by
 $$\Delta_{\mathbb{H}^k} :
=\sum_{i=1}^{2k} X_i^2 =  \Delta_{x} +4|x|^{2} \partial_{yy} +\sum_{i=1}^{k}( 4x_{i+k}\partial_{yx_i}- 4x_{i}\partial_{yx_{i+k}}) .$$
The Lane-Emden equation of the Heisenberg Laplacian  is then formulated as:
\begin{equation}\label{Yamabe}
  -\Delta_{\mathbb{H}^k} u = |u|^{p-1}u,  ~~ \text{in}~ \mathbb{R}^{2k} \times \mathbb{R}. 
  \end{equation}

Recently, Wei-Wu \cite{wuke} constructed   cylindrical singular solutions to equation \eqref{Yamabe}. 
Here, we define a function 
 $u(x,y)$  as cylindrical if it is radially symmetric in the variable $x$. 
It is straightforward to observe that for a cylindrical function $u$, we have $x_{i+k}\partial_{x_i}u- x_{i}\partial_{x_{i+k}}u=0$,  thus  reducing $\Delta_{\mathbb{H}^k}  u= \Delta_{x}u +4|x|^{2} \partial_{yy}u$. 
As a consequence of  Theorem \ref{thm1}, we obtain the following corollary:

\begin{corollary}\label{cor}
 If  $p\in (1, p_{\text{S}}(2k+2)) \cup  (p_{\text{S}}(2k+2), p_{\text{JL}}(2k+2))$ , $u=0$ is the unique $C^2$ cylindrical  solution to \eqref{Yamabe} that is stable outside a compact set.
\end{corollary}

When $p=p_{\text{S}}(2k+2)$, equation \eqref{Yamabe} transforms into the CR Yamabe problem. Without restricting to cylindrical solutions, Jerison and Lee \cite{Jerison1989,Jerison1988} proved that,  up to group translations and dilations, a suitable multiple of the function 
\begin{equation}\label{1.5a}
u(x,y)=\frac{1}{((1+|x|^2)^2+y^2)^{\frac{2k}{4}}} 
\end{equation}
is the unique positive entire solution of the CR Yamabe problem.  It is evident that the function 
 $u$  in \eqref{1.5a} is  cylindrical. 
Moreover, if $p=p_{\text{S}}(Q)$, for any $m\geq 1, n\geq 1$, and  $\alpha\geq 0$,  equation \eqref{01} admits a  cylindrical 
solution $u$  with an asymptotic estimate at infinity  (see Dou-Sun-Wang \cite{DSWZ2022} and Chen-Liao-Zhang \cite{clz}). 
  Consequently, the Cwikel-Lieb-Rozenblum inequality (c.f. Levin-Solomyak \cite{Levin}) ensures that $u$ possesses a finite Morse index, making $u$  stable outside a compact set.

Unfortunately, the existence of a nontrivial $C^2$ stable solution (or stable outside a compact set) remains unknown for cases where  $\alpha>0, Q\geq11$ and $p\geq p_{\text{JL}}(Q)$. In fact, when $\alpha>0$, from \eqref{2-11},
 equation \eqref{01} cannot be solved as an ordinary differential equation, and it is natural to consider the cylindrical solutions.  

Our paper is organized as follows:
\begin{enumerate}[$\bullet$]
    \item In the subsequent section, we present essential preliminaries concerning the Baouendi-Grushin operator.
    \item Section 3 and 4 employ the Moser iteration method (c.f. \cite{Farina}) to establish (1) in Theorem \ref{thm1}. Additionally, an asymptotic estimation at infinity is provided in Section 5. 
    \item Section 6 is dedicated to formulating a monotonicity formula, and  we complete the remaining proof in Section 7 by applying the monotonicity formula.
\end{enumerate}

\section{Preliminaries of Baouendi-Grushin operator}
In our notation, a point $z=(x,y)\subset \mathbb{R}^N = \mathbb{R}^m \times \mathbb{R}^n$  is represented as follows:
\[z=(x,y)= (x_1,\ldots,x_m,y_1,\ldots, y_n). \]

The group of non-isotropic dilations $\{\delta_{r}\}_{r>0}$  is defined by
\begin{equation}
  (x,y)\longmapsto \delta_{r}(x,y):=(r x, r^{\alpha+1}y ).
\end{equation}
The infinitesimal generator of the family of dilations is given by the vector field
\[Zu(z)=\lim_{r\to 1} \frac{u(\delta_r(z)) -u(z)}{r-1} =\sum_{1}^{m}x_{i}\partial_{x_i}u+\sum_{1}^{n}(\alpha+1)y_{i}\partial_{y_i}u.\]
Such vector field is characterized by the property: $u(\delta_{r}(\cdot))=r^su(\cdot)$ if and only if $Zu=su$.

As is customary, we denote the homogeneous dimension with respect to $\{\delta_{r}\}_{r>0}$ as $Q=m+ n(\alpha +1)$, 
and use the notation $\nabla_{\alpha} := (\nabla_{x} , (\alpha+1)|x|^{\alpha}\nabla_{y})$ to denote the Grushin gradient.
One can check that
\begin{equation}\label{2.444}
\nabla_{\alpha} ( u( \delta_{r}(\cdot)))=r\nabla_{\alpha} u (\delta_{r}(\cdot)),~\Delta_{\alpha} ( u( \delta_{r}(\cdot)))=r^2 \Delta_{\alpha} u (\delta_{r}(\cdot)).
\end{equation}

It is well known that the following Sobolev-type inequality does hold,
\begin{equation}
  \|u\|^2_{L^{\frac{2Q}{Q-2}}(\mathbb{R}^N)}\leq C\|\nabla_{\alpha}u\|^2_{L^2(\mathbb{R}^N)}~~~~\forall u\in C_0^{1}(\mathbb{R}^N).
\end{equation}

Moreover, set
\begin{equation}\label{eq:heisenberg-norm}
\rho(x,y)=(|x|^{2(\alpha+1)}+|y|^{2})^{\frac{1}{2(\alpha+1)}}.
\end{equation}
The ball and  sphere with respect to $\rho$ centered at the origin with radius $r$ are
\begin{equation}
B_r =\{(x,y) \mid \rho(x,y)<r \},~\partial B_r =\{(x,y) \mid \rho(x,y)=r \}.
\end{equation}

We present a set of useful formulas, primarily derived through direct calculations, also see Garofala \cite{Garofala1993}.

\begin{equation}\label{2-11}
\Delta_{\alpha}  f(\rho)=\frac{|x|^{2\alpha}}{\rho^{2\alpha}}\left(f''(\rho)+\frac{Q-1}{\rho}f'(\rho)\right).
\end{equation}

\begin{equation}\label{2-12}
    \int_{B_R\setminus B_r}f(\rho(z)) \,dz = Q|B_1| \int_{r}^{R}\rho^{Q-1}f(\rho) \,d\rho.
\end{equation}

\begin{equation}\label{2-13}
    [\nabla_{\alpha}, Z]=\nabla_{\alpha}Z-Z\nabla_{\alpha}=\nabla_{\alpha}, \quad \text{div~} Z=Q.
\end{equation}

\begin{equation}\label{2-14}
    \quad |\nabla_{\alpha}\rho|^2=\frac{|x|^{2\alpha}}{\rho^{2\alpha}}, \quad \nabla_{\alpha}\rho\cdot \nabla_{\alpha}u=\frac{|x|^{2\alpha}Zu}{\rho^{1+2\alpha}}.
\end{equation}

We conclude this section with the following Pohozaev identity:

\begin{proposition}
    Suppose $u$ is a $C^2$ solution to $-\Delta_{\alpha} u=f(u)$ on $\mathbb{R}^N$, with $F(u)=\int_{0}^{u}f(s)\,ds$, then for any $\varphi\in C_0^{\infty}(\mathbb{R}^N)$, the Pohozaev identity holds:
    \begin{equation}\label{Pohozaev identity}
        \begin{aligned}
            \int_{\mathbb{R}^N} &\Big{[}\frac{Q-2}{2}|\nabla_{\alpha} u|^2-QF(u) \Big{]}\varphi \,dz \\
            &= \int_{\mathbb{R}^N} \Big{[}F(u)-\frac{1}{2}|\nabla_{\alpha} u|^2\Big{]}Z\varphi \,dz+ \int_{\mathbb{R}^N} \nabla_{\alpha} u\cdot \nabla_{\alpha} \varphi Zu \,dz.
        \end{aligned}
    \end{equation}
\end{proposition}

\begin{proof}
  We test the equation by $\varphi Zu$, then
  \begin{equation}\label{2.15}
  \begin{aligned}
    \int_{\mathbb{R}^N} \nabla_{\alpha}u \cdot \nabla_{\alpha}(\varphi Zu )dz& =  \int_{\mathbb{R}^N} f(u)\varphi Zu dz=\int_{\mathbb{R}^N} \varphi ZF(u) dz\\&
  =-\int_{\mathbb{R}^N} F(u) Z\varphi dz- \int_{\mathbb{R}^N} Q F(u)\varphi dz.
  \end{aligned}
  \end{equation}
  On the other hand, by \eqref{2-13},
    \begin{equation}\label{2.16}
  \begin{aligned}
    \int_{\mathbb{R}^N} \nabla_{\alpha}u \cdot \nabla_{\alpha}(\varphi Zu )dz
    &=\int_{\mathbb{R}^N} (\nabla_{\alpha}u \cdot \nabla_{\alpha}\varphi  Zu + \varphi \nabla_{\alpha}u \cdot \nabla_{\alpha}Zu) dz\\&
  =\int_{\mathbb{R}^N} (\nabla_{\alpha}u \cdot \nabla_{\alpha}\varphi Zu+ \varphi \nabla_{\alpha}u \cdot Z\nabla_{\alpha}u + \varphi \nabla_{\alpha}u \cdot \nabla_{\alpha}u) dz\\&
  =\int_{\mathbb{R}^N} ( \nabla_{\alpha}u \cdot \nabla_{\alpha}\varphi Zu+\frac{1}{2} \varphi Z|\nabla_{\alpha}u|^2 + \varphi|\nabla_{\alpha}u|^2) dz\\&
  =\int_{\mathbb{R}^N} ( \nabla_{\alpha}u \cdot \nabla_{\alpha}\varphi Zu- \frac{1}{2} Z \varphi|\nabla_{\alpha}u|^2+\frac{2-\text{div~}Z}{2} \varphi |\nabla_{\alpha}u|^2) dz.
    \end{aligned}
  \end{equation}
 Combining \eqref{2.15} and \eqref{2.16}, we arrive at \eqref{Pohozaev identity}.
  \end{proof}

  $\mathbf{Notations}$: Throughout the paper, different positive constants are commonly denoted by $C$ or $C(\alpha,\beta)$ if they are related to certain parameters $\alpha$ and $\beta$, and may vary as needed. The symbol $dz$ represents the Lebesgue measure on $\mathbb{R}^{N}$, while $d\sigma$ denotes the $(N-1)$-dimensional Hausdorff measure on a surface in $\mathbb{R}^{N-1}$.

\section{Moser Iteration}
Before presenting the proofs, it is necessary to recall the following:
\begin{definition}
We say that a $C^2$ solution $u$ to \eqref{01},
\begin{enumerate}
    \item is stable if
    \begin{equation}
    \int_{\mathbb{R}^N} |\nabla_{\alpha} \varphi|^2 \,dz \geq \int_{\mathbb{R}^N}  p|u|^{p-1}\varphi^2 \,dz, \quad \forall \varphi \in C_c^{1}(\mathbb{R}^N).
    \end{equation}

    \item is stable outside a compact set $K$ if
    \begin{equation}\label{stable}
    \int_{\mathbb{R}^N} |\nabla_{\alpha} \varphi|^2 \,dz \geq \int_{\mathbb{R}^N}p|u|^{p-1}\varphi^2 \,dz, \quad \forall \varphi \in C_c^{1}(\mathbb{R}^N\setminus K).
    \end{equation}

    \item has Morse index equal to $k$ if $k$ is the maximal dimension of a subspace $W_k$ of $C_c^{1}(\mathbb{R}^N)$ such that
    \[\int_{\mathbb{R}^N} |\nabla_{\alpha} \varphi|^2 \,dz < \int_{\mathbb{R}^N} p|u|^{p-1}\varphi^2 \,dz, \quad \forall \varphi \in W_k\setminus \{0\}.\]
\end{enumerate}
\end{definition}
\begin{remark}
  Any finite Morse index solution $u$ is stable outside a compact set. 
\end{remark}
The next lemma follows from 
the Moser iteration agruments presented in \cite{Farina},
for the  sake of completeness, we restate it:
\begin{lemma}\label{prop}
  If $u$ is a $C^2$ solution to \eqref{01}, which is stable outside a compact set $K$, and suppose $p>1$. Define $s(p)=2p+2\sqrt{p(p-1)}-1$. Then
   for  $s\in [1, s(p))$ and any integer $m\geq \frac{s+p}{p-1}$, there exists a positive constant $C(s,p,m)$ such that
\begin{equation}\label{08}
   \int_{\mathbb{R}^N} |u|^{p+s} \phi^{2m} dz  \leq C(s,p,m) \int_{\mathbb{R}^N} (|\nabla_{\alpha} \phi|^2 + |\phi \Delta_{\alpha}\phi|)^{\frac{s+p}{p-1}} dz
\end{equation}
holds for all $\phi \in C_c^{\infty}(\mathbb{R}^N\setminus K)$ with $0\leq \phi \leq 1$.
\end{lemma}
\begin{proof}
Testing  equation \eqref{01} with $|u|^{s-1} u \varphi^2$, where $ \varphi \in C_c^{\infty}(\mathbb{R}^N\setminus K)$ and $s\geq 1$, we obtain
\begin{equation}\label{03}
\begin{aligned}
 \int_{\mathbb{R}^N} &|u|^{p+s} \varphi^2 dz
  = \int_{\mathbb{R}^N} \nabla_{\alpha} u \cdot \nabla_{\alpha} (|u|^{s-1} u \varphi^2)dz\\
  &
  =\frac{4s}{(s+1)^2} \int_{\mathbb{R}^N} |\varphi \nabla_{\alpha}(|u|^{\frac{s-1}{2}}u)|^2dz+ \frac{1}{s+1} \int_{\mathbb{R}^N}\nabla_{\alpha}(|u|^{s+1}) \cdot \nabla_{\alpha}(\varphi^2)dz
  \\&
  =\frac{4s}{(s+1)^2} \int_{\mathbb{R}^N} |\varphi \nabla_{\alpha}(|u|^{\frac{s-1}{2}}u)|^2dz- \frac{1}{s+1} \int_{\mathbb{R}^N}|u|^{s+1}\Delta_{\alpha}(\varphi^2)dz.
\end{aligned}
\end{equation}
On the other hand, selecting $|u|^{\frac{s-1}{2}}u\varphi$ as a test function in the stability inequality \eqref{stable} yields
\begin{equation}\label{04}
   \begin{aligned}
    \int_{\mathbb{R}^N}& p|u|^{p+s} \varphi^2 dz \leq
    \int_{\mathbb{R}^N} |\nabla_{\alpha} (|u|^{\frac{s-1}{2}}u\varphi)|^2 dz\\
   &
   = \int_{\mathbb{R}^N}( |\varphi \nabla_{\alpha} (|u|^{\frac{s-1}{2}}u)|^2+|u|^{s+1}|\nabla_{\alpha} \varphi|^2+\frac{1}{2} \nabla_{\alpha}(|u|^{s+1}) \cdot \nabla_{\alpha}(\varphi^2))dz\\&
   = \int_{\mathbb{R}^N}( |\varphi \nabla_{\alpha} (|u|^{\frac{s-1}{2}}u)|^2+|u|^{s+1}|\nabla_{\alpha} \varphi|^2-\frac{1}{2}|u|^{s+1}\Delta_{\alpha} (\varphi^2))dz.
   \end{aligned}
 \end{equation}
Note that $s(p)=2p+2\sqrt{p(p-1)}-1$,  it can be verified that $4ps(p)=(s(p)+1)^2$. Subtracting \eqref{04} from \eqref{03}, we observe that,
for $s\in [1, s(p))$ ,
   \begin{equation}\label{3.6}
    \int_{\mathbb{R}^N} |u|^{p+s} \varphi^2 dz \leq  C(s,p) \int_{\mathbb{R}^N} (|u|^{s+1}|\nabla_{\alpha} \varphi|^2+|u|^{s+1}|\Delta_{\alpha} (\varphi^2)|) dz.
 \end{equation}

Let $\varphi=\phi^m$ for some  integer $m\geq 1$ and $\phi \in C_c^{\infty}(\mathbb{R}^N\setminus K)$ such that $0\leq \phi \leq 1$. From \eqref{3.6}, we obtain:
\begin{equation}\label{06}
 \int_{\mathbb{R}^N} |u|^{p+s} \phi^{2m} dz\leq
C(s,p,m) \Big{(} \int_{\mathbb{R}^N}  |u|^{s+1}\phi^{2m-2}|\nabla_{\alpha} \phi|^2 dz+  \int_{\mathbb{R}^N}|u|^{s+1}\phi^{2m-2}
|\phi\Delta_{\alpha} \phi| dz \Big{)}.
\end{equation}
When $m\geq \frac{s+p}{p-1}$, it holds that $(2m-2)\cdot \frac{s+p}{s+1}\geq2m$. Thus, $\phi^{(2m-2)\cdot \frac{s+p}{s+1}}\leq \phi^{2m}$, and by H\"{o}lder inequality, we have:
\begin{equation}\label{07}
\begin{aligned}
  \int_{\mathbb{R}^N} |u|^{p+s} \phi^{2m} dz &\leq
C(s,p,m) \Big{(} \int_{\mathbb{R}^N}|u|^{p+s}\phi^{(2m-2)\cdot \frac{s+p}{s+1}} dz \Big{)}^{\frac{1+s}{p+s}}\Big{[} \int_{\mathbb{R}^N} (|\nabla_{\alpha} \phi|^2+ |\phi \Delta_{\alpha}\phi|)^{\frac{s+p}{p-1}} dz \Big{]}^{\frac{p-1}{s+p}}
\\&\leq
C(s,p,m) \Big{(} \int_{\mathbb{R}^N}|u|^{p+s} \phi^{2m} dz \Big{)}^{\frac{1+s}{p+s}}\Big{[} \int_{\mathbb{R}^N} (|\nabla_{\alpha} \phi|^2+ |\phi \Delta_{\alpha}\phi|)^{\frac{s+p}{p-1}} dz \Big{]}^{\frac{p-1}{s+p}}
.
\end{aligned}
\end{equation}
Therefore, for $s\in [1, s(p))$, we conclude that
\begin{equation}
   \int_{\mathbb{R}^N} |u|^{p+s} \phi^{2m} dz  \leq C(s,p,m) \int_{\mathbb{R}^N} (|\nabla_{\alpha} \phi|^2 + |\phi \Delta_{\alpha}\phi|)^{\frac{s+p}{p-1}} dz.
\end{equation}
\end{proof}
\begin{remark}
   If $u$ is stable, for $r>0$, let
\begin{equation}\label{phi}
  \phi(z)=\phi(\rho(z))=\begin{cases}
           1, & \mbox{if } z \in B_r \\
            0, & \mbox{if } z \in B_{2r}^c.
         \end{cases}
\end{equation}
Then we have $|\nabla_{\alpha}\phi|\leq \frac{C(Q)}{r}$ and $|\Delta_{\alpha}\phi|\leq \frac{C(Q)}{r^2}.$
Hence, by choosing $K=\emptyset$,
 from \eqref{08} and \eqref{2-12}, we must have
\begin{equation}\label{09}
\begin{aligned}
\int_{B_{r}} |u|^{p+s} dz
&\leq C(s,p) \int_{\mathbb{R}^N} (|\nabla_{\alpha} \phi|^2 + |\phi \Delta_{\alpha}\phi|)^{\frac{s+p}{p-1}} dz\\&
\leq  C(s,p,Q)\int_{B_{2r}\backslash B_{r}} r^{-\frac{2(s+p)}{p-1}} dz\\&
\leq C(s,p,Q)r^{Q-2\cdot \frac{s+p}{p-1}}.
\end{aligned}
\end{equation}
If we further assume $p \in (1, p_{\text{JL}}(Q))$,  there exists $s_0\in [1, s(p))$ such that $Q-2\cdot \frac{s_0+p}{p-1}<0$. Letting $r\to\infty$, the integral $\int_{\mathbb{R}^N}|u|^{p+s_0} dz$ equals to 0, implying $u=0$.
\end{remark}
\section{Proof of (1)  in Theorem 1.1}
\begin{proof} [Proof of (1) in Theorem \ref{thm1}]
 Without loss of generality, we assume that $u$ is stable outside $B_{1}$. For $p\in (1,p_{\text{S}}(Q))$ and $s\in [1, s(p))$, where $Q<\frac{2(1+p)}{p-1}$, we choose the cut-off function $\phi$ with $r>2$ as follows:
\begin{equation}\label{1-13}
 \phi(z)= \phi(\rho(z))=\begin{cases}
        0,  & \mbox{if } \rho(z)\in [0,1] \\
        1, & \mbox{if } \rho(z)\in [2,r]\\
        0, & \mbox{if }  \rho(z) \in [2r,\infty].
      \end{cases}
\end{equation}
Then we immediate obtain from \eqref{08} that
\begin{equation}\label{1-13}
\begin{aligned}
  \int_{B_r\setminus B_2} |u|^{s+p} dz&
  \leq C(s,p)\int_{\mathbb{R}^N} (|\nabla_{\alpha} \phi|^2 + |\phi \Delta_{\alpha}\phi|)^{\frac{s+p}{p-1}} dz\\&
  \leq C(s,p,Q) ( \int_{B_{2}\setminus B_{1}} dz + r^{-\frac{2(s+p)}{p-1}}\int_{B_{2r}\setminus B_{r}}  dz )
  \\& \leq
   C(s,p, Q)(1+r^{Q-\frac{2(s+p)}{p-1}}).
\end{aligned}
\end{equation}
Similarly, for $r$ large enough, by choosing suitable cut-off function in \eqref{08},  the following holds:
\begin{equation}\label{4.5}
\int_{B_{2r}\backslash B_r} |u|^{s+p} dz\leq C(s,p,Q)r^{Q-\frac{2(s+p)}{p-1}}.
\end{equation}
Since $u\in C^2$, it follows that $u\in L^{s+p}(\mathbb{R}^N)$.
Choosing $ \varphi(z)= \varphi(\rho(z))$ such that:
\begin{equation}\label{var}
  \varphi(z)=\begin{cases}
           1, & \mbox{if } z \in B_r \\
            0, & \mbox{if } z \in B_{2r}^c.
         \end{cases}
\end{equation}
Testing  equation \eqref{01} with $u\varphi^2$ we obtain that
\begin{equation}\label{4.7}
 \int_{\mathbb{R}^N} |u|^{p+1} \varphi^2 dz
  = \int_{\mathbb{R}^N} |\varphi \nabla_{\alpha}u|^2dz- \frac{1}{2} \int_{\mathbb{R}^N}u^{2}\Delta_{\alpha}(\varphi^2)dz.
  \end{equation}
Since
\begin{align*}
\left|\int_{\mathbb{R}^N} u^{2}\Delta_{\alpha}(\varphi^2) dz\right|
&\leq C\int_{B_{2r}\backslash B_{r}} \frac{u^2}{r^2} dz\\&
\leq Cr^{Q-2-\frac{2Q}{1+p}}\left(\int_{B_{2r}\backslash B_{r}} |u|^{1+p} dz\right)^{\frac{2}{1+p}}
\\&\leq Cr^{Q-\frac{2(p+1)}{p-1}}
 \xrightarrow{r\to +\infty}0,
\end{align*}
we conclude from \eqref{4.7} and  $|u|^{p+1}\in L^1(\mathbb{R}^N)$ that $|\nabla_{\alpha}u|^{2}\in L^1(\mathbb{R}^N)$
and
\begin{equation}\label{2-8}
 \int_{\mathbb{R}^N} |u|^{p+1} dz
 =\int_{\mathbb{R}^N}  | \nabla_{\alpha} u|^2 dz.
\end{equation}
For $\varphi$ defined in \eqref{var}, applying the Pohozaev identity \eqref{Pohozaev identity} and letting $r\to +\infty$, we obtain:
  \begin{equation}\label{2-6}
    \frac{Q-2}{2}\int_{\mathbb{R}^N} |\nabla_{\alpha} u|^2 dz=\frac{Q}{p+1}\int_{\mathbb{R}^N}|u|^{p+1}dz.
  \end{equation}
  We point out here that we have used \eqref{2-14}  to obtain
\begin{equation*}
  \begin{aligned}
  |\nabla_{\alpha} u\cdot \nabla_{\alpha} \varphi Zu|&\leq 
  |\varphi'(\rho)||\nabla_{\alpha} u\cdot\nabla_{\alpha} \rho| |Zu|\\&
  \leq \frac{C}{\rho} |\nabla_{\alpha} u| |\nabla_{\alpha} \rho| |Zu|\\&
  \leq C\frac{\rho^{2\alpha}}{|x|^{2\alpha}} |\nabla_{\alpha} u|^2|\nabla_{\alpha} \rho|^2\\&\leq C|\nabla_{\alpha} u|^2,
  \end{aligned}
\end{equation*}  
  and hence
  \[\int_{\mathbb{R}^N} |\nabla_{\alpha} u\cdot \nabla_{\alpha} \varphi Zu| \,dz
   \leq C\int_{B_{r}^c} |\nabla_{\alpha} u|^2\,dz
   \to 0\text{~as~}r\to+\infty.\]
Therefore,
\begin{equation}\label{2-9}
  0 =(\frac{Q}{p+1}-\frac{Q-2}{2})\int_{\mathbb{R}^N} |u|^{p+1} dz.
\end{equation}
In  the case  $ p\in (1, p_{\text{S}}(Q))$, we must have $\int_{\mathbb{R}^N} |u|^{p+1} dz=0$, thus $u=0$.
\end{proof}
\section{Asymptotic Analysis}
Now we have established the validity of (1) in Theorem \ref{thm1}. The following asymptotic estimates will be useful in the remaining proof:
\begin{lemma}\label{infinity}
 For $p\in [p_{\text{S}}(Q), p_{\text{JL}}(Q))$, if $u$ is a $C^2$ solution to $-\Delta_{\alpha} u=|u|^{p-1}u$ and is stable outside a compact set, then \[\lim_{\rho(z)\to+\infty} \rho(z)^{\frac{2}{p-1}}|u(z)| = 0.\]

\end{lemma}
\begin{proof}
 Without loss of generality, we assume once more that $u$ is stable outside $B_{1}$.
Since $p\in [p_{\text{S}}(Q), p_{\text{JL}}(Q))$, it holds that $\frac{Q(p-1)}{2}= p+s_1$ for some $s_1\in  [1, s(p))$, and
there exist $\eta>0$ and $\bar{s}\in (1, s(p))$ such that
$(p-1)\frac{Q+\eta}{2}=p+\bar{s}$. Similar to \eqref{1-13}, we observe that for $r>2$,
  \[\int_{B_r\backslash B_2} |u|^{\frac{Q(p-1)}{2}} dz=\int_{B_r\backslash B_2} |u|^{p+s_1} dz\leq C(p, Q).\]
  Thus, $u\in L^{\frac{Q(p-1)}{2}}(\mathbb{R}^N)$. Then for any $\epsilon>0$, there exists  $R=R(\epsilon)>2$ such that
  \[  \int_{B_{R}^c} |u|^{\frac{Q(p-1)}{2}} dz\leq \epsilon.\]
  Moreover, similar to \eqref{4.5}, for any $r>R$, we have
  \[  \int_{B_{4r}\setminus B_r} |u|^{p+\bar{s}} dz \leq Cr^{Q-\frac{2(p+\bar{s})}{p-1}}\leq Cr^{-\eta}. \]
 We define the function $u_r(z):=r^{\frac{2}{p-1}}u( \delta_{r}(z))$. Therefore, we have
  $$-\Delta_{\alpha} u_r=|u_r|^{p-1}u_r,~ \text{in~} B_4\setminus B_1.$$
  Note that for  $r>R$, $$\int_{B_4\setminus B_1} |u_r|^{(p-1)\frac{Q+\eta}{2}}dz
  = r^{\eta} \int_{B_{4r}\setminus B_r} |u|^{p+\bar{s}} dz\leq C.$$
  Then by Lemma \ref{moser} and H\"{o}lder inequality,  we obtain
  \begin{equation}
  \begin{aligned}
  \|u_r\|_{L^{\infty}(B_3\setminus B_2)}&\leq C \|u_r\|_{L^2(B_4\setminus B_1)}
  \\&
  \leq  Cr^{\frac{2}{p-1}-\frac{Q}{2}}\Big{(} \int_{B_{4r}\setminus B_r} |u|^{2} dz \Big{)}^{\frac{1}{2}}\\&
  \leq C\Big{(}  \int_{B_{4r}\setminus B_r} |u|^{\frac{Q(p-1)}{2}} dz\Big{)}^{\frac{2}{Q(p-1)}}\\&
  \leq C\Big{(} \int_{B_{R}^c} |u|^{\frac{Q(p-1)}{2}} dz\Big{)}^{\frac{2}{Q(p-1)}}\\&
  \leq C\epsilon^{\frac{2}{Q(p-1)}}.
  \end{aligned}
  \end{equation}
  Taking $r\to+\infty$ and $\epsilon \to 0$, we conclude the proof.
\end{proof}

The following lemma applies a standard Moser iteration, and we provide a sketch of the proof.
\begin{lemma}\label{moser}
If $v$ is a  $C^2$ solution to $-\Delta_{\alpha}v=c(z)v$ in $B_4\setminus B_1$, and $c\in L^{q}(B_4\setminus B_1)$ with $q>\frac{Q}{2}$, then for some positive constant $C=C(q,N,\|c\|_{ L^{q}(B_4\setminus B_1)})$, we have
\begin{equation}
  \|v\|_{L^{\infty}(B_3\setminus B_2)}\leq C \|v\|_{L^2(B_4\setminus B_1)}.
\end{equation}
\end{lemma}
\begin{proof}
Testing the equation with $|v|^{s-1} v \varphi^2$, where $ \varphi \in C_c^{\infty}(B_4\setminus B_1)$ is a cut-off function and $s\geq 1$, we obtain
by Young's inequality:
\begin{equation}
\begin{aligned}
\int_{\mathbb{R}^N} | \nabla_{\alpha}(|v|^{\frac{s-1}{2}}v \varphi)|^2dz  \leq C s (\int_{\mathbb{R}^N} c|v|^{s+1} \varphi^2 dz
 + \int_{\mathbb{R}^N} |v|^{s+1}|  \nabla_{\alpha} \varphi|^2dz ).
\end{aligned}
\end{equation}
Denoting $w=|v|^{\frac{s-1}{2}}v$,  by Sobolev inequality and $c\in L^{q}(B_4\setminus B_1)$, we get
\begin{equation}
\begin{aligned}
 \|w \varphi\|_{L^{\frac{2Q}{Q-2}}(\mathbb{R}^N)}^2&  \leq
   C\int_{\mathbb{R}^N} | \nabla_{\alpha}(w \varphi)|^2dz  \\&\leq Cs (\int_{\mathbb{R}^N} cw^2 \varphi^2 dz
 + \int_{\mathbb{R}^N} w^2|  \nabla_{\alpha} \varphi|^2dz )\\&
 \leq C s (\|w \varphi\|_{L^{\frac{2q}{q-1}}(\mathbb{R}^N)}^2+  \int_{\mathbb{R}^N} w^2|  \nabla_{\alpha} \varphi|^2dz)
 . 
\end{aligned}
\end{equation}
By interpolation  inequality with $\frac{2Q}{Q-2}>\frac{2q}{q-1}>2$ if $q>\frac{Q}{2}$, we have
\[ \|w \varphi\|_{L^{\frac{2q}{q-1}}(\mathbb{R}^N)}^2 \leq  \epsilon   \|w \varphi\|_{L^{\frac{2Q}{Q-2}}(\mathbb{R}^N)}^2 + C\epsilon^{-\frac{Q}{2q-Q}}  \|w \varphi\|_{L^{2}(\mathbb{R}^N)}^2. \]
Therefore, 
\begin{equation}\label{w}
  \|w \varphi\|_{L^{\frac{2Q}{Q-2}}(\mathbb{R}^N)}^2 \leq Cs^{\frac{2q}{2q-Q}} \int_{\mathbb{R}^N}  (\varphi^2+|  \nabla_{\alpha} \varphi|^2) w^2  dz.
\end{equation}
Denote $S_{j}:=B_{4-\sum_{k=1}^{j}\frac{1}{2^k}}\setminus B_{1+\sum_{k=1}^{j}\frac{1}{2^k}}$ with $S_0:=B_4\setminus B_1$, and let
 $\varphi_{j}$ be a cut-off function supported in $S_j$   and equals to 1 in $S_{j+1}$. Set $\gamma_{j}:=2(\frac{Q}{Q-2})^j$.  
Note that $|  \nabla_{\alpha} \varphi_j|\leq C2^j$.  By choosing $s=\gamma_{j}-1$ and $\varphi=\varphi_{j}$  in \eqref{w}, we obtain:
\[ \|v\|_{L^{\gamma_{j+1}}(S_{j+1})}    \leq C^{\frac{j}{\gamma_{j}}} \|v\|_{L^{\gamma_{j}}(S_{j})}, ~~\forall j\geq 0 .\]
Thus,
\begin{equation}
\begin{aligned}
\|v\|_{L^{\infty}(B_3\setminus B_2)}\leq C^{\sum_{j=0}^{+\infty} \frac{j}{\gamma_{j}} } \|v\|_{L^2(B_4\setminus B_1)}
\leq C\|v\|_{L^2(B_4\setminus B_1)}.
\end{aligned}
\end{equation}
\end{proof}
\section{Monotonicity Formula}
If $u$ is  a $C^2$  solution  to $-\Delta_{\alpha} u=|u|^{p-1}u$,  for $r>0$, we define the energy functional:
$$E(r, u)=r^{\frac{2(p+1)}{p-1}-Q}\Big{(}\frac{1}{2}\int_{B_r}|\nabla_{\alpha} u|^2 dz-\frac{1}{p+1}\int_{B_r} |u|^{p+1} dz\Big{)}.$$
Recall $u_{r}(x,y)=r^{\frac{2}{p-1}}
u(r x, r^{1+\alpha}y)$, and  one can verify that $-\Delta_{\alpha} u_{r}=|u_{r}|^{p-1}u_{r}$ and $E(r, u)=E(1, u_{r})$.
Moreover, since the unit outer normal vector $\nu$ on the boundary $\partial B_1$ is given by $\frac{\nabla \rho(z)}{|\nabla \rho(z)|}$,  integrating by parts and utilizing \eqref{2-14}, we obtain the following equality:
\begin{equation}
  \begin{aligned}
  \frac{dE(r, u)}{dr}&=\frac{dE(1, u_r)}{dr}\\
  &= \int_{B_1}  \nabla_{\alpha} u_r \cdot \nabla_{\alpha}  \frac{du_r}{dr} dz+\int_{B_1}\Delta_{\alpha}u_r  \frac{du_r}{dr}  dz\\&
  =\int_{\partial B_1}  \frac{du_r}{dr} \frac{\nabla_{\alpha}  \rho}{|\nabla \rho|} \cdot \nabla_{\alpha} u_r d\sigma
  =\int_{\partial B_1} \frac{|x|^{2\alpha}}{|\nabla \rho|} \frac{du_r}{dr} Zu_{r} d\sigma.
  \end{aligned}
\end{equation}
Note that
\[ \frac{du_r}{dr} =\frac{1}{r}(\frac{2}{p-1}u_r+Zu_r ).\]
Thus,
\begin{equation}
  \begin{aligned}
   \frac{dE(r, u)}{dr}=\int_{\partial B_1} \frac{r|x|^{2\alpha}}{|\nabla \rho|} (\frac{du_r}{dr})^2 d\sigma-\frac{d}{dr}\int_{\partial B_1} \frac{|x|^{2\alpha}u_r^2}{(p-1)|\nabla \rho|} d\sigma.
   \end{aligned}
\end{equation}
In particular, we conclude that:
\begin{theorem}[Monotonicity formula]
For $$I(r,u):=E(r, u)+\int_{\partial B_1} \frac{|x|^{2\alpha}u_r^2}{(p-1)|\nabla \rho|} d\sigma,$$
 we have
\begin{equation}\label{monotonicity formula}
\frac{d I(r,u)}{dr}=\int_{\partial B_1} \frac{|x|^{2\alpha}}{r|\nabla \rho|} (\frac{2}{p-1}u_r+Zu_r )^2 d\sigma\geq 0.
\end{equation}
\end{theorem}
\section{Proof of (2) in Theorem 1.1}

\begin{proof}[Proof of (2) in Theorem \ref{thm1}]
Let $u$ be a $C^2$ solution to $-\Delta_{\alpha} u=|u|^{p-1}u$ that is stable outside a compact set.
By Lemma \ref{infinity},
  \[\lim_{\rho(z)\to+\infty} \rho(z)^{\frac{2}{p-1}}|u(z)| = 0.\]
  As a consequence, we have $\lim\limits_{r\to+\infty}\sup\limits_{\rho(z)=1}|u_r(z)|=0$, therefore,
  \[\lim\limits_{r\to+\infty}\int_{\partial B_1} \frac{|x|^{2\alpha}u_r^2}{(p-1)|\nabla \rho|} d\sigma=0. \]
  And for any $\epsilon>0$, there exists $R>0$ such that for $z \in B_R^c$, we have $$\rho(z)^{\frac{2}{p-1}}|u(z)|\leq \epsilon.$$ Since $p\in (p_{\text{S}}(Q), p_{\text{JL}}(Q))$, then $\frac{2(p+1)}{p-1}-Q<0$ and
\begin{equation}\label{ppp}
 \begin{aligned}
r^{\frac{2(p+1)}{p-1}-Q} \int_{B_r} |u|^{p+1} dz
 & \leq  Cr^{\frac{2(p+1)}{p-1}-Q} (\int_{B_R}|u|^{p+1} dz+ \epsilon r^{Q-2\frac{(p+1)}{p-1}})\\&
  \leq   C(r^{\frac{2(p+1)}{p-1}-Q}\int_{B_R}|u|^{p+1} dz+ \epsilon).
\end{aligned}
\end{equation}
Thus, $$\lim\limits_{r\to+\infty} r^{\frac{2(p+1)}{p-1}-Q} \int_{B_r} |u|^{p+1} dz=0.$$

On the other hand, since $ -\Delta_{\alpha} u_r=|u_r|^{p-1}u_r$,  we test this equation with $\varphi^2 u_r$, where \begin{equation*}
  \varphi(z)=\begin{cases}
           1, & \mbox{if } z \in B_1 \\
            0, & \mbox{if } z \in B_{2}^c.
         \end{cases}
\end{equation*}
We obtain 
\begin{equation*}
  \begin{aligned}
  \int_{B_1}|\nabla_{\alpha} u_r|^2 dz&\leq  \int_{\mathbb{R}^N}  \varphi^2|\nabla_{\alpha} u_r|^2 dz\\&
  \leq C \int_{\mathbb{R}^N} ( |\nabla_{\alpha} \varphi|^2|u_r|^2 +  \varphi^2 |u_r|^{p+1})dz\\&
  \leq C(\int_{B_2} |u_r|^{p+1} dz)^{\frac{2}{p+1}}  + C\int_{B_2} |u_r|^{p+1} dz.
  \end{aligned}
\end{equation*}
Similar to \eqref{ppp}, we have $\lim\limits_{r\to+\infty}\int_{B_2} |u_r|^{p+1} dz=0$, hence
\[ \lim\limits_{r\to+\infty} r^{\frac{2(p+1)}{p-1}-Q}\int_{B_r}|\nabla_{\alpha} u|^2 dz =\lim\limits_{r\to+\infty} \int_{B_1}|\nabla_{\alpha} u_r|^2 dz  =0. \]
Therefore,
\[ \lim\limits_{r\to+\infty} I(r, u)=0.\]
Since $u$ is $C^2$, we conclude that $\lim\limits_{r\to0} I(r, u)=0$. Due to the monotonicity of $I(r,u)$, we obtain that $$I(r, u)\equiv 0 ~\mbox{  and   }~ \frac{2}{p-1}u+Zu=0,$$ hence $u$ is a homogeneous solution with degree $-\frac{2}{p-1}$, i.e. $$u(\delta_{r}(z))=r^{-\frac{2}{p-1}}u(z),\forall z\in\mathbb{R}^N.$$ Therefore, the continuity of $u$ yields that $u= 0$.
\end{proof}

\section*{Acknowledgements}
Hua Chen is supported by National Natural Science Foundation of China (Grant Nos. 12131017, 12221001) and National Key R\&D Program of China (no. 2022YFA1005602). 

\section*{Declarations}
On behalf of all authors, the corresponding author states that there is no conflict of interest. Our manuscript has no associated data.

\end{document}